\documentclass[11pt]{amsart}
\usepackage{times}
\usepackage[T1]{fontenc}
\usepackage{amssymb, amsthm, amsmath}
\usepackage{mathrsfs}

\usepackage[shortlabels]{enumitem}

\usepackage[active]{srcltx}

\usepackage{todonotes}

\usepackage{setspace}
\usepackage{geometry,soul}

\usepackage{hyperref}

\newtheorem{theorem}{Theorem}[section]

\newtheorem{lemma}[theorem]{Lemma}

\newtheorem{proposition}[theorem]{Proposition}

\newtheorem*{gen-dif}{\fbox{{\large A}} \hypertarget{Agen-dif}{Gen-Dif}}

\newtheorem*{min-balln}{\fbox{{\large A}} \hypertarget{Amin-ball}{Cballs}}

\theoremstyle{definition}
\newtheorem{definition}[theorem]{Definition}
\newtheorem{example}[theorem]{Example}

\newtheorem{situation}[theorem]{Situation}

\renewcommand{\phi}{\varphi}

 \title{Closed bounded sets in 1-h-minimal valued fields}
 \author{Juan Pablo Acosta L\'opez}

\date{\today}

\begin{document}

\maketitle
\begin{abstract}
We show that the 1-h-minimal fields satisfy a property of naive compactness
for decreasing definable 
families of closed bounded sets indexed by the value group.
We use this to prove that a local topological definable group has a definable
family of neighborhoods of the identity consisting of open subgroups.
\end{abstract}
\section{Introduction}
In this paper we show that a 1-h-minimal field with algebraic
$RV$, satisfies a naive property of compactness, 
where every decreasing definable 
family
of closed bounded sets in $K^n$ 
has nonempty intersection provided the family is indexed
by the value group $\Gamma$. This is Proposition \ref{main-compactness}.
Algebraic $RV$ in residue characteristic $0$ or 
$RV_{p,\bullet}$ in mixed characteristic, means that the 
definable subsets of $RV$, respectively $RV_{p,\bullet}$,
are the same as those of the language of pure valued fields.
In the language of pure valued fields there is a quantifier elimination
relative to these sorts so we can describe these sets in terms of a
suitable language for them. An example of a language with this
property is the extension of the algebraic language by restricted
analytic functions. Every 1-h-minimal field remains 1-h-minimal when
the language in $RV$ is expanded, so in some sense the property of 
1-h-minimality is a property relative to $RV$. For arbitrary
$RV$ this property of compactness is not true.

This property does not distinguish between, say, open balls and closed
balls in an algebraically closed valued field, so it is different from
notions of properness in rigid geometry.

We give as an application of this result that every definable local topological
group in a 1-h-minimal valued field with algebraic $RV$ has a fundamental
neighborhood of open subgroups. This is analogous to the fact that a 
$p$-adic local Lie group has such a fundamental neghborhood too.

The proof of the main result in residue characteristic $0$ 
uses the fact that a decreasing family of bounded sets in an
ordered linear group, (the value group), is eventually constant. 
This property of ordered linear groups is proven in Section \ref{secordgro}.
This
follows from a quantifier elimination result of ordered linear
groups relative to a certain family of colored linear orders, the spine
of the group. This was done in \cite{CluHal}. For the theory of colored linear
orders we give a cell decomposition like result to show what is needed.
Just like in other contexts, like the real closed fields or the p-adically closed fields, 
this cell decomposition actually implies previously known
quantifier elimination results for colored linear orders, see for instance
\cite{SimDPMinOrd}. We mention that this method of proof has the advantage
over the more usual back-and-forth that the resulting quantifier elimination
is generally constructive, in the sense that one gets a primitive
recursive map from formulas to equivalent quantifier free formulas. 

In Section \ref{seccom}, we prove the compactness property
in residue characteristic $0$. We work in a slightly more general framework
designed to be used in the proof in mixed characteristic.

In Section \ref{seccommix}, we prove the compactness property
in mixed residue characteristic. For this we work in a residue
characteristic $0$ coarsening of the given valuation. Unfortunately
the resulting $RV$ is not algebraic, 
hence the need for the axiomatic
framework in Section \ref{seccom}.

In Section \ref{seclocgro} we give our application to local definable
topological groups. The arguments follow topological lines using
the compactness property.

In the Appendix we present the needed background on ind-definable
sets for the construction appearing in the cell decomposition for 
colored linear orders. The terminology is unusually categorical, 
which we hope is clarifying for readers familiar with categories.
\section{Notation}
Here we will use basic notions of model theory without remark.
The models will generally be taken to be $\omega$-saturated.

A valued field $K$ with valuation $v$ and value group $\Gamma$
will be denoted additively $v:K\to \Gamma\cup\{\infty\}$.
We denote $\mathcal{O}$ the valuation ring of $K$ and 
$\mathcal{M}$ the maximal ideal of $\mathcal{O}$.
We denote $RV=K^{\times}/(1+\mathcal{M})$ and $rv:K^{\times}\to RV$ the 
canonical projection. 

We will use the notion of a $n$-h-miminal valued field, most often when
$n=1$, defined
in residue characteristic $0$ in \cite{hensel-min} and in mixed characteristic 
$(0,p)$ in \cite{hensel-minII}.

\section{Ordered abelian groups}\label{secordgro}
The main result of this section is Proposition 
\ref{closed-bounded-ordered-group}, which is needed for the next 
section.

In ordered linear groups there is a quantifier elimination relative to 
a collection of interpretable colored linear orders (linear orders with
a finite number of predicates).

We start by giving a quantifier elimination for colored linear orders which
is somewhat more detailed than in \cite{SimDPMinOrd} Proposition 4.1.

Suppose we have a linear order $S$ with a set of predicates $\mathcal{U}$.

Given an element $a\in S$ and a predicate $U$, we can define 
$U_-(a)=\text{sup}\{x\in U: x<a\}$ and 
$U_{+}(a)=\text{inf}\{x\in U:x>a\}$.
These two elements might not necessarily exist in $S$, but they
give definable Dedekind cuts in $S$. 
So we get $U_-,U_+:S\to \bar{S}$, where $\bar{S}$ is the set of definable
cuts of $S$, together with a maximal or minimal element if $S$ does not
have them.
See the Appendix \ref{ind-definable}
for the precise definition of $\bar{S}$.
Note that the formulas for $U_-$ and $U_+$ also make sense when $a\in \bar{S}$,
so we can form the compositions. 
More precisely $\bar{S}$ is a strict ind-definable set
(actually because the language does not have two distinct constants
it is a colimit of finite disjoint unions of interpretable sets), 
and $U_{+},U_{-}:\bar{S}\to \bar{S}$ becomes an ind-definable map. 
See Example \ref{completion-ind-definable} for the definition of the 
ind-definable structure in $\bar{S}$ and
 Example \ref{completion-function-definable} for the verification that 
$U_+$ and $U_-$ are ind-definable maps.

We denote $F(\mathcal{U})$ the set of functions 
$F:\bar{S}\to \bar{S}$
which are a finite composition of functions of the form $U_+$ and $U_-$ for 
$U\in\mathcal{U}$, the functions $U_+$ and $U_-$ are the basic functions in
$F(\mathcal{U})$.

\begin{lemma}\label{u-calculus}
\begin{enumerate}
\item
The functions in $F(\mathcal{U})$ are non-decreasing, and satisfy
$U_-(a)\leq a\leq U_+(a)$.
Also $U_-U_+(a)\leq a\leq U_+U_-(a)$.
\item
For $a$ and $b$ in $\bar{S}$ we have that there is 
$x\in S$ with $x\in U$ and $a<x<b$ if and only if 
$U_{+}(a)<b$. 
\item $U_+(a)<b$ if and only if $a<U_{-}(b)$. 
\item $b\leq U_{+}(a)$ if and only if $U_-(b)\leq a$. 
\item
 $b=U_{+}(a)$ is equivalent to $U_{+}U_{-}(b)=b$ and 
$U_{-}(b)\leq a$.
\item $a=U_-(b)$ is equivalent to $U_-U_+(a)=a$ and $U_+(a)\geq b$.
\end{enumerate}
\end{lemma}
The proof is a straightforward check which we omit.

The motivation of the statement of items 3,4, 5 and 6 is to express
the conditions $U_+(x)<b$ and $b<U_+(x)$ as a Boolean combination of 
conditions involving only inequalities of $x$ and functions on $b$,
and similarly for $U_-$.

For the next definition we use the following terminology
\begin{definition}
Let $X$ be a set, $\mathcal{A}$ be a collection of subsets of $X$ and 
$B$ be a subset of $X$. We say that $\mathcal{A}$ refines $B$ if for every
$A\in\mathcal{A}$ either $A\subset B$ or $A\subset X\setminus B$.

If $\mathcal{B}$ is a collection of subsets of $X$, then $\mathcal{A}$ refines
$\mathcal{B}$ if it refines every element of $\mathcal{B}$.
\end{definition}
Starting from a finite set of colors $\mathcal{U}$ one can inductively define
$\mathcal{U}_n$ consisting of a finite partition of $S$ into 
$0$-definable unary relations such that:
\begin{itemize}
\item $\mathcal{U}_0$ refines $\mathcal{U}$ and $\mathcal{U}_{n+1}$ refines 
$\mathcal{U}_n$.
\item If $F$ and $G$ are in 
$F(\cup_{k<n}\mathcal{U}_k)$ and are a composition
of less than $n$ basic functions, 
then the set $\{x:F(x)<G(x)\}$
is refined by the partition $\mathcal{U}_n$.
\item If $F$ is as in the previous item and $V\in \cup_{k<n}\mathcal{U}_k$ 
then
the set of $x$ such that $F(x)\in S$ and $F(x)\in V$ 
is refined by the partition $\mathcal{U}_n$.
\end{itemize}
Then we have a new set of colors $\mathcal{U}^c=\cup_n\mathcal{U}_n$ (no longer
necessarily finite).

In other words $\mathcal{U}^c$ is such that its finite unions are closed
under boolean combinations and that one can decide the truth value of 
$F(x)<G(x)$, by looking at the conditions $x\in V$ for $V\in \mathcal{U}^c$.

The fact that $\{x\in S\mid F(x)<G(x)\}$ is $\emptyset$-definable in $S$ 
follows from Example \ref{completion-order-definable}.
The fact that $\{x\in S\mid F(x)\in S\text{ and } F(x)\in U\}$ is
 $\emptyset$-definable follows from Example \ref{completion-ind-definable}.

Then we have

\begin{proposition}\label{qe-order}
Suppose $(S,\mathcal{U})$ is a colored linear order.
Then every $\emptyset$-definable set of $S^n$ is a boolean combination 
of sets of the form 
$\{a\in S^n:\pi(a)\in U\}$ for $U\in \mathcal{U}^c$, or of the form
$\{a\in S^n: F\pi(a)<G\pi'(a)\}$,
where $F$ and $G$ are functions in $F(\mathcal{U}^c)$,
and $\pi,\pi'$ denote either a coordinate projection or a constant function
equal to the minimum or maximum (in $\bar{S}$)

Further for every $b$-definable family $\{D_a\}_{a\in T}$ of subsets of $S$, there is a finite
$b$-definable partition of $T$ such that in each one
$D_a$ is a disjoint finite union of sets of the form
$\{x:L\pi(a,b)<x<R\pi'(a,b), x\in U\}$ or of the form
$\{x:E\pi(a,b)=x\}$,
 where here $L,R,E,\pi,\pi'$ and $U$ are as above.

Also for every $b$-definable function $f:X\to S$ with $X\subset S^n$ we have
that there is a finite $b$-definable partition of the domain $X$ such that
in each one $f$ is of the form $f(a)=F(\pi(a,b))$.
\end{proposition}
\begin{proof}
We prove the first two statements at the same time with an
explicit elimination procedure.

A set defined as in the first statement will be called 
quantifier free definable.
Let $D_a\subset S$ be an $a$-definable set, where 
$D=\{(x,a)\mid x\in D_a\}$ is quantifier free definable.
 We show first that there is a finite partition of the set $T$
of the parameters $a$ 
into quantifier free definable sets,
such that $D_a$ has the form of the 
second statement for $a$ in each of the elements in the partition.
After this we show that the set of $a$ such that there exists $x$ with 
$(x,a)\in D$, is equivalent to a quantifier free definable set in $a$.

Note that there is a partition of $T$ of sets of the required form such that 
for $a$ in each element of the partition  we have that $D_a$ 
consists of a finite disjoint union of sets which are finite intersections
of sets of one of the following forms:
\begin{enumerate}
\item
$L\pi(a)<G(x)$, 
\item
$G(x)<R\pi(a)$,
\item
$G(x)=E\pi(a)$,
\item
$G(x)<F(x)$
\item
$G(x)=F(x)$
\item
$G(x)\in U$ with $U\in\mathcal{U}^c$.
\item $G(x)\notin S$.
\end{enumerate}
The last four cases are equivalent to a finite disjoint union of sets
of the form $x\in U$ with $U\in \mathcal{U}^c$.

The first three cases are each equivalent to a finite disjoint union
of conditions of the same form with $G=1$, 
and of admissible conditions on
$a$. This is a consequence of 
Lemma \ref{u-calculus}, 
(and induction in the length of the composition in $G$).
So we assume $G=1$ from now on.

In case  3 appears $D_a$ is empty or consists of a single element as required,
depending on quantifier free conditions on $a$.
(according to whether $L\pi'(a)<E\pi(a)$ or not for example).

If case 3 does not appear we can further divide $a$ into 
admissible cases to pick
out the largest of the lower bounds in case 1, 
and the smallest of the upper bounds in case 2,
and so we obtain that $D_a$ is of the desired form.

Now we show that the set of $a$ with $D_a$ not empty is 
quantifier free definable.
 This finishes the proof of the first and second
statements. This is a consequence of Lemma \ref{u-calculus}, item 2.

To see the statement on definable functions we have to see that if 
$\{x:L\pi(a)<x<R\pi'(a), x\in U\}$
consists of only one point then the set is a singleton of the form 
required. In this case one has 
$x=U_{+}L\pi(a)$.
\end{proof}
\begin{proposition}\label{closed-bounded-linear-order}
Suppose $S$ is a colored linear order. Assume we have a definable family
$D_r\subset S^n$ indexed by $r\in S$ which is, in each coordinate, uniformly 
bounded above. 
Then there is $r_0$ and a finite definable partition of $\{r: r\geq r_0\}$
such that in each element of the partition  $D_r$ is constant.
\end{proposition}
\begin{proof}
The statement is trivial if $S$ has a maximum, so assume it does not.

First assume $n=1$. 
Suppose $d$ is an upper bound for $D_r$.
From the description of definable sets in one dimension in 
Proposition \ref{qe-order} we obtain that the result easily from the 
following observation:

If $F\in F(\mathcal{U}^c)$
then either $d\ll F(a)$ for every $a\in \bar{S}$ such that $d\ll a$,
or $F(a)=F(a')$ for all $a,a'$ such that $d\ll a,a'$.
Here the notation $d\ll a$ means $G(d)< a$ for all 
$G\in F(\mathcal{U}^c)$ such that $G(d)$ is not $\infty$.

This condition on $F\in F(\mathcal{U}^c)$ 
is preserved under compositions so we just have to prove it
for $F=U_+$ and $F=U_-$. The first case 
follows as $a\leq U_+(a)$.
The second case follows from Lemma \ref{u-calculus},
indeed $G(d)<U_-(a)$ is equivalent to $U_+G(d)<a$ 
so this condition can only fail when $U_+G(d)=\infty$, which means
that $U\subset [-\infty,G(d)]$ and so its supremum is $U_-(a)$ which does
not depend on $a$.

For the case $n>1$ 
consider the the family 
$\{D_{ba}\}_{b\in (-\infty,d)^{n-1}}$ of one dimensional sets, defined as
$D_{ba}=\{x\in S\mid (x,b)\in D_a\}$. By 
the argument above one obtains
a finite partition into definable
subsets of $(-\infty,d)^{n-1}\times S$ such that for $E$ in the partition, $(b,a),(b,a')\in E$ and
$d\ll a, a'$ we have $D_{ba}=D_{ba'}$. By induction on $n$ the families $E_a$ are definably piecewise constant
eventually, and this finishes the proof.
\end{proof}

\begin{proposition}\label{bounded-functions-ordered-group}
Let $\Gamma$ be an ordered commutative group, and let
$E_a$ be a bounded definable family of finite sets indexed by $\Gamma$.
Then for $a$ sufficiently large and definable piecewise in $a$, $E_a$ is
constant.
\end{proposition}
\begin{proof}
We may assume $E_a\subset \Gamma$ and that they have the same cardinality. 
Because $\Gamma$ comes with an order we may assume this cardinality is 1,
so $E_a=\{f(a)\}$. Definably piecewise in $a$ the function $f$ is affine,
see Corollary 1.10 of \cite{CluHal}
and because it is bounded it is constant in the unbounded pieces.
\end{proof}
\begin{proposition}\label{closed-bounded-ordered-group}
Let $\Gamma$ be an ordered commutative group, and let
$E_a\subset \Gamma$ be a decreasing definable family 
of bounded sets indexed by 
$a\in \Gamma$. Then there is $a_0$ such that if $a>a_0$ then 
$E_a=E_{a_0}$.
\end{proposition}
\begin{proof}
This proof becomes simpler if $\Gamma$ has finite spines.

We will use the relative quantifier elimination of \cite{CluHal}.
We use the language $L_{syn}$ defined there. Note that the Theorem 1.13
of \cite{CluHal} implies that $L_{syn}$ eliminates quantifiers relative
to the spine. (In fact it is equivalent to that, see Proposition 1.11 of
\cite{CluHal}).

We will show that for $a$ divisible (by every natural) 
and sufficiently large $E_a$ is constant. As the family is decreasing
this is enough.

We actually show the stronger property 
that if the family is uniformly bounded then there
 is a finite definable partition of the set of $a$, such that
for $a$ divisible and sufficiently 
large in one of the elements of the partition then
$E_a$ is constant. (The condition of $a$ being divisible is necessary for 
this stronger property).
Equivalently, if $M=(-d,d)$ is some fixed interval then
$M\cap E_a$ satisfies the required condition for any family $E_a$.
Note that this reformulation is closed under boolean combinations.

We have that the set
$E$ is given by a boolean combination of 
 atomic formulas in $\Gamma$ or inverse images
of definable subsets of the spine by a term in the language. So we may
assume $E$ is either an atomic formula in $\Gamma$ or the inverse image
of a definable subset of the spine by a term.

Assume first we have $\rho(b,a)$ a term in $a,b$ into the spine, and 
$D$ a definable subset of the spine such that $E=\rho^{-1}(D)$.
Replacing $D$ by the image under $\rho$ we may assume $E\to D$ is surjective.

Note that $\rho$ is a tuple of functions of the form 
$\mathfrak{s}_{p^r}(ra+sb+c)$, and $\mathfrak{t}_p(ra+sb+c)$,
where $c$ are constants
which define the family $E_a$, and $r$ and $s$ are natural numbers.
Note that $\mathfrak{s}_{p^r}(ra+sb+c)=\mathfrak{s}_{p^r}(sb+c)$ because 
$a$ is divisible.
Note also that for $a$ sufficiently large one has 
$\mathfrak{t}_p(ra+sb+c)=\mathfrak{t}_p(a)$.
We conclude that $\rho(a,b)=(\rho'(b),\mathfrak{t}_p(a))$ and so
$E_a=\rho'^{-1}D_{\mathfrak{t}_p(a)}$. Finally note that 
$E_a$ being uniformly bounded implies that 
$D_{\mathfrak{t}_p(a)}=\rho'(E_a)$ is uniformly bounded above, so we finish
by the Proposition \ref{closed-bounded-linear-order}.

Now assume that $E$ is given by an atomic formula in the sort $\Gamma$.
If $E$ is given by the condition $ra<sb+c$ then the intersection with any
interval fixed interval becomes empty, the whole interval, or does not
depend on $a$ for $a$ large, depending on whether $r>0$, $r<0$ or $r=0$.
Now assume $E$ is given by the condition $ra+sb+c\equiv_m 0$. 
Then because $a$ is divisible, this condition does not depend in $a$.
Similarly we have the condition $ra+sb+c\equiv_m k_{\bullet}$ does not depend on 
$a$. We are left with the condition $ra+sb+c=k_{\bullet}$. For $a$ sufficiently
large and divisible by $k+1$ this conditions never occurs. This finishes
the proof.
\end{proof}
\section{Definable compactness}\label{seccom}
Nowak shows, for example,
 that in a 1-h-minimal field of residue characteristic $0$ with 
algebraic RV the image of a closed bounded set under a continuous definable
map is closed and bounded in \cite{Nowak} Proposition 5.4.
In this section we give a different proof of this fact.
In fact this proof works in a somewhat more general setting described next.
This is not an idle generalization as at the moment 
the mixed characteristic case in the unbounded ramification case
needs the slight extra generality.

\begin{situation}\label{closed-bounded-h0}
Here we consider a valued field $K$ such that
\begin{enumerate}
\item $K$ is definably spherically complete. In other words
any definable family of balls which form a chain has a nonempty intersection.
\item The map $rv:K^{\times}\to RV$ is opaque.
In other words for every definable set $X\subset K^{\times}$ the set
$rv(X)\cap rv(K^{\times}\setminus X)$ is finite.
\item Every definable set $S\subset k\times \Gamma$ is a boolean combination
of cartesian products of sets definable in $k$ and in $\Gamma$.
\item Every definable subset of $\Gamma^n$ is definable in the ordered
group language of $\Gamma$.
\end{enumerate}
\end{situation}
Conditions 1 and 2 are true in a 0-h-minimal field of residue characteristic
$0$, (if $C$ prepares $X$ then the set mentioned in 2 is contained in $C$,
and if $C$ prepares the family of balls in 1, then an element of $C$ 
will be in the intersection).
Condition 3 and 4 are true in any henselian valued field of 
residue characteristic $0$ with the 
valued field language. 
This follows for example from Denef-Pas quantifier elimination, see for 
instance \cite{pas89}.
This is a quantifier elimination relative to $RV$ where 
$K$ comes equipped with a cross section of
$RV\to \Gamma$, in
the ordered group language in 
$\Gamma$, the field language in $k$ and the angular component map
$K^{\times}\to k^{\times}$ coming from the cross section.
In this case $\Gamma$ and $k$ are orthogonal because there are no relation
or function symbols connecting the two sorts.

For convenience of reference of the next section we make this argument
explicit.
\begin{lemma}\label{orthogonal}
Suppose that we have a model $M$ 
in a many sorted relational language $\mathcal{L}$, and suppose
that the collection of sorts is partitioned into two sets $\mathcal{A}$
and $\mathcal{B}$. We consider $\mathcal{L}_\mathcal{A}$ and 
$\mathcal{L}_\mathcal{B}$ the restrictions of $L$ to these sorts.
Suppose every basic relation $R\subset A\times B$ (or a coordinate 
permutation of $A\times B$) is a boolean 
combination of cartesian products of sets definable in $A$ in
$\mathcal{L}_{\mathcal{A}}$ and sets
definable in $B$ in $\mathcal{L}_{\mathcal{B}}$, where here
$A$ is a product of the sorts in $\mathcal{A}$ and $B$ is a product of the
sorts in $\mathcal{B}$. 
Then every definable set $X\subset A\times B$ is a boolean combination
of a cartesian product of sets definable in $A$ in $\mathcal{L}_{\mathcal{A}}$
and sets definable in $B$ in $\mathcal{L}_{\mathcal{B}}$.
\end{lemma}
The proof is an easy induction on the quantifier complexity of $X$.

Note also that all the conditions in Situation \ref{closed-bounded-h0}
are true in $ACVF$ (in any characteristic),
and also their analytic expansions.
Typically however condition 2 does not hold for valued fields in 
mixed characteristic. We deal with them later.
\begin{proposition}\label{closed-bounded-0}
Suppose $K$ is a valued field as in Situation \ref{closed-bounded-h0}.
Then the intersection of a 
decreasing family of nonempty bounded closed
sets indexed by $\Gamma$ is nonempty.
\end{proposition}
\begin{proof}
Suppose we have a fixed family
$\{C_r\}_r$
 like that with empty intersection.
First we prove that $rv(C_r)$ is eventually constant equal to a finite set.

Note that $v(C_r)$ is a bounded decreasing family, and so by 
Proposition \ref{closed-bounded-ordered-group}
it is eventually constant equal to some set $R$.
On the other hand if $E_r\subset RV$ is the exceptional set of $C_r$,
that is $rv(C_r)\cap rv(K^{\times}\setminus C_r)$,
then $v(E_r)$ is bounded definable family of finite sets, and so definably
piecewise in $r$ it is eventually constant, 
by Proposition \ref{bounded-functions-ordered-group}.
If $R$ is infinite we conclude that there is a
$t\in R$ not in $v(E_r)$ for all $r$ sufficiently large large.
Note that
$J_t=v^{-1}(t)\subset RV$ is definably isomorphic to $k^{\times}$, 
so by orthogonality $rv(C_r)\cap J_t$ is eventually constant.
Any point in the eventual set has an $rv$ fiber in the intersection of $C_r$.
We conclude that $v(C_r)$ is eventually constant equal to a finite set.

If $t\in v(C_r)$, then, 
as noted before $J_t\cap rv(C_r)$ is eventually constant,
also $E_r\cap J_t$ is eventually constant definably piecewise in $r$,
so if the eventual value of $J_t\cap rv(C_r)$ is infinite then any point
in this eventual value and not in $E_r\cap J_t$ for $r$ large has its
$rv$ fiber in the intersection. We conclude that $rv(C_r)$ is eventually
constant and finite, as desired.

Let us call a decreasing family unbranched if
whenever it is contained in a finite union of disjoint open balls eventually,
then it is contained in one of these open balls eventually.
Now we prove $\{C_r\}_r$ is eventually contained in a finite union of open
balls such that the intersection of the family with each one is unbranched.
If not then we can form an infinite tree with finite levels, each point
corresponding to an open ball,
branching at each point not a leaf,
such that points later in the tree correspond to smaller open balls
and incomparable nodes correspond to disjoint open balls, the 
family is eventually contained in the union of the balls corresponding
to a maximal antichain at finite levels, 
and the family has nonempty intersection with
all balls in the tree. By König's Lemma we may find an infinite
chain in the tree. By $\omega$-saturativity we may assume there is a point
in the intersection, and translating we may assume this point is $0$.
If $B$ is a ball in the tree not in the chain, then the valuation of $B$
is in the eventual value of $v(C_r)$, so we conclude that this is infinite
in contradiction with the above.

So now assume the family is unbranched. Take the family of open balls
that eventually contain the family. By definable spherical completeness
there is a point in the intersection, and again we may assume this point is
$0$. As $v(C_r)$ is eventually constant equal to a finite set we conclude
that this family of open balls actually intersect in a closed ball around
the origin (of radius the minimum of the eventual set of $v(C_r)$).
As the $rv$ values are also eventually constant we conclude that $C_r$
is eventually contained in a union of a finite number of disjoint open balls
of smaller radiuses, the unbranched hypothesis now says that it is 
eventually contained in one of these balls, and this contradicts the 
choice of the closed ball as the minimal containing $C_r$ eventually.
\end{proof}
After this, standard topological arguments about compactness go through, 
we write this next.
For convenience of reference we interpolate the conclusion of the previous
proposition as a hypothesis
\begin{situation}\label{closed-bounded-h1}
In this situation we consider a valued field $K$ satisfying the conclusion
of the previous proposition, that is, such that
every decreasing definable family 
of closed bounded nonempty subsets of $K$ indexed
by $\Gamma$ has nonempty intersection.
\end{situation}
\begin{proposition}\label{closed-bounded-0.5}
Let $K$ be as in Situation \ref{closed-bounded-h1}, then
every decreasing definable family of closed bounded nonempty subsets of 
$K^n$ indexed by $\Gamma$, has a nonempty intersection.
\end{proposition}
\begin{proof}
This is by induction on the dimension of the ambient space $n$. Suppose
it is true for $n$ and that the family $\{C_r\}_r$ lives in $K^n\times K$.
Denote $\pi_1:K^n\times K\to K^n$ the first projection and 
$\pi_2:K^n\times K\to K$ the second projection.
By the induction
hypothesis one has a point $a\in \cap_r\overline{\pi_1(C_r)}$.
Now consider the family $D_{r,s}=\overline{\pi_2(C_r\cap \pi_1^{-1}(B_s(a))}$.
This is a decreasing family of nonempty closed sets so by 
the hypothesis
(applied to $D_{r,r}$)
 we get a point $b$ in the intersection. Now a straightforward verification
shows $(a,b)\in \cap_r(C_r)$.
\end{proof}
\begin{proposition}
Suppose $K$ is as in Situation \ref{closed-bounded-h1}.
Suppose $X\subset K^n$ is closed and bounded and $f:X\to K^m$ is definable
and continuous. Then $f(X)$ is closed and bounded.
\end{proposition}
\begin{proof}
To prove that it is bounded consider the family $C_r=\{x\in X\mid v(f(x))<-r\}$.
To see that it is closed consider the family $f^{-1}(B_r(a))$ for 
$a\in\overline{f(X)}$.
\end{proof}
\begin{proposition}\label{wallman}
In Situation \ref{closed-bounded-h1}. Assume $f:X\times Y\to K^n$ is 
continuous and definable, for $X\subset K^r$ and  $Y\subset K^s$.
Suppose $L\subset X$ and $S\subset Y$ are closed and bounded 
in $K^r$ and $K^s$ respectively.
Suppose $W$ is a definable open set with $f(L\times S)\subset W$.
Then there are definable open sets $U\subset K^r$ and $V\subset K^s$ such that
$L\subset U$, $S\subset V$ and $f(U\cap X\times V\cap Y)\subset W$.
\end{proposition}
\begin{proof}
Otherwise we can consider the family
 $\overline{f^{-1}(K^n\setminus W)}\cap B_r(L)\times B_r(S)$ to obtain a contradiction
\end{proof}
\section{The closedness theorem in mixed characteristic}\label{seccommix}
In this section we prove the closedness theorem in the form of 
Proposition \ref{closed-bounded-0} for 1-h-minimal valued fields
of mixed characteristic with algebraic $RV_{p,\bullet}$.
Here we are denoting $RV_{p,\bullet}$ the collection of sorts
$RV_n=K^{\times}/(1+p^n\mathcal{M})$ where $\mathcal{M}$ is the maximal
ideal of the valuation ring of $K$.
 The proof
proceeds by verifying that the conditions in Situation \ref{closed-bounded-0}
apply to $K$ with a characteristic $0$ coarsening of the valuation.
Denote $\mathcal{L}_c$ the given language in $K$ expanded with such a 
coarsening. Then $K_c$ (that is $K$ considered as a valued field with the 
coarse valuation $v_c$) is 1-h-minimal in the language $\mathcal{L}_c$, so 
it satisfies hypothesis 1 and 2.
In this section we verify it also satisfies hypothesis 3 and 4.
The fact that the any residue characteristic $0$ coarsening is 1-h-minimal
is ine of the main theorems of \cite{hensel-minII}, see Theorem 2.2.7.
\begin{proposition}
If $K$ is a 0-h-minimal mixed characteristic field then $RV_{p,\bullet}$ 
is stably 
embedded, in the strong sense of Proposition 2.6.12 of \cite{hensel-min}.
\end{proposition}
\begin{proof}
The proof of Proposition 2.6.12 in \cite{hensel-min} , 
using Lemma 2.3.1 and Proposition 2.3.2 in 
\cite{hensel-minII}
goes through.
\end{proof}
\begin{proposition}\label{resplendent-rv}
If $K$ is a valued field which is $0$-h-minimal 
then there is a definitionally equivalent language in which there is
 resplendent relative elimination of quantifiers to 
$RV_{p,\bullet}$ or $RV$ according to the residue characteristic.
\end{proposition}
\begin{proof}
We expand the language and assume that every $0$-definable
set in $RV_{p,\bullet}$ is quantifier free definable. Also we may assume 
that every $0$-definable function $f:K^n\to RV_{p,\bullet}$ is a term.
We also assume that $RV_{p,\bullet}$ form a closed family of sorts in the 
terminology of \cite{rideau-2017} Appendix A.
Then by the result of that Appendix we just have to see 
elimination of quantifiers.

It is enough to see that every formula $\exists x\varphi(x,y,z)$ where 
$y$ is a tuple of field variables, $x$ is a field variable and $z$ is an
$RV_{p,\bullet}$ variable; 
is equivalent to a quantifier free formula.
If we denote by $A_{yz}=\{x\mid \vDash \varphi(x,y,z)\}$,
and $D_{yz}=rv(A_{yz})$ then $A_{yz}$ is not empty if and only if $D_{yz}$
is not empty. By the strong stable embededness we conclude that 
$D_{yz}=D'_{t(y)z}$ for some term $t$ 
(Apply strong embededness to $D_y=\cup_z D_{yz}\times \{z\}$).
\end{proof}

\begin{definition}
Suppose $K$ is a valued field of residue characteristic $0$.
In $RV$ we consider the following language. 
$RV$ comes equipped with the group language and 
sorts for $k^{\times}$ and $\Gamma$, and function
symbols for the short exact sequence 
$1\to k^{\times}\to K^{\times}\to \Gamma\to 0$. Additionally $\Gamma$ has the
ordered group language and $k^{\times}$ has a ternary relation symbol 
for addition (wherever it is defined).

We call this language the Basarab language in $RV$.

If $K$ is a valued field of mixed characteristic we consider in 
$RV_{p,\bullet}$ the language that comes with one sort for each $RV_n$ with
the group language, and
sorts for $\mathcal{R}_n=\mathcal{O}^{\times}/(1+p^n\mathcal{M})$
and $\Gamma$, function symbols for 
$1\to \mathcal{R}_n\to RV_n\to \Gamma\to 0$, 
and the order in $\Gamma$.
Additionally we have a function symbol for the projection $RV_{n+1}\to RV_n$.
Finally, in each $RV_n$ we have a ternary relation symbol 
$\oplus$ and a binary relation symbol $\ominus$,
defined as $\oplus(x,y,z)$ if one has 
$rv_n(x')=x, rv_n(y')=y$, $rv_n(z')=z$
and $z'=x'+y'$, for some $x',y',z'$; and $\ominus(x,y)$ if there is and
$x'\in K$ with $rv_n(x')=x$ and $rv(-x')=y$.

We say the field $K$ has algebraic $RV_{p,\bullet}$ if every 
$0$-definable subset
of $RV_{n}^m$ in $K$ is definable in $RV_{p,\bullet}$ in the Basarab language.
\end{definition}
Every henselian mixed characteristic valued field $K$ with the pure
valued field language has algebraic $RV_{p,\bullet}$, see for example
\cite{flenner-2011}. Also every henselian valued field in 
mixed characteristic with an analytic expansion also has algebraic 
$RV_{p,\bullet}$, see for instance Theorem 6.3.7 of \cite{CluLip}.
\begin{lemma}\label{rv0}
Suppose $K$ is a valued field of residue characteristic $0$.
Consider in $RV$ the Basarab language $\mathcal{L}_b$.
In $RV$ denote the relations $\oplus$ and $\ominus$ defined as
$\oplus(x,y,z)$ if $rv(x')=x,rv(y')=y, rv(z')=z$ for some $x,y,z$ with 
$x'+y'=z'$ and $\ominus(x,y)$ if $rv(x')=x, rv(-x')=y$ for some 
$x'$. Then $\oplus$ and $\ominus$ definable in $\mathcal{L}_b$.
\end{lemma}
\begin{proof}
Recall that the relation $\oplus^r(x,y,z)$ defined as 
$x,y,z\in k^{\times}$ and $z=x+y$ is definable in the Basarab language.
From this we get that $\ominus^r(x,y)$ defined as $y=-x, x,y\in k^{\times}$
is also definable by the condition 
$x,y\in k^{\times}\land \lnot(\exists z)\oplus^r(x,y,z)$.

To define $\oplus$ we consider cases.
If $v(x)<v(y)$ then $\oplus(x,y,z)$ is equivalent to $z=x$.
Similarly if $v(y)<v(x)$, $\oplus(x,y,z)$ is equivalent to $z=y$.
If $v(x)=v(y)$, we get $\oplus(x,y,z)$ is equivalent to 
$\oplus(1,yx^{-1},zx^{-1})$ so assume $x=1$ and $v(y)=0$.
Now if $y\neq -1$ we have that $\oplus(1,y,z)$ is equivalent to 
$\oplus^r(1,y,z)$.
Finally $\oplus(1,-1,z)$ is equivalent to $v(z)>0$.

The analysis for $\ominus(x,y)$ is similar. Namely,
$\ominus(x,y)$ implies $v(x)=v(y)$, and if $v(x)=v(y)$ it is equivalent to
$\ominus(1,yx^{-1})$, which is equivalent to $\ominus^r(1,yx^{-1})$.
\end{proof}
\begin{lemma}\label{orthogonal-p}
Let $K$ be a valued field of mixed characteristic in a language $\mathcal{L}$.
Suppose we have a non-trivial residue characteristic $0$ coarsening $v_c$.
In $RV_c$ denote $\mathcal{L}_{be}$ the Basarab language expanded by 
the maps $RV_c\to RV_{p,\bullet}$, and the Basarab language in $RV_{p,\bullet}$.
Then $k_c$ and $\Gamma_c$ are orthogonal, that is, every $0$-definable
subset of $k_c^m\times \Gamma_c^n$ is a boolean combination of 
cartesian products of $0$-definable sets in $\Gamma_c^n$ and $k_c^m$.

Also every $0$-definable subset of $\Gamma_c^n$ is $0$-definable in the 
ordered group language.
\end{lemma}
Note that $\mathcal{R}_n$ and $\Gamma$ are
are not orthogonal in the Basarab language when the field $K$ is henselian of
 infinite 
ramification. This is because $\mathcal{R}_1$ is a definable quotient of
$\mathcal{R}_n$ and there is a surjective function 
$\mathcal{R}_1\setminus\{1\}\to [0,v(p))$ in the valued field language
(and so in the Basarab language)
given by $w(x)=v(1-x)$.
\begin{proof}
We consider a diferent language in $RV_c$, denoted $\mathcal{L}_1$  
consisting of the Basarab language
in $RV_c$ together with
a predicate for $\mathcal{O}^{\times}/(1+\mathcal{M}_c)\subset k_c^{\times}$, 
and predicates for
$(1+p^n\mathcal{M})/(1+\mathcal{M}_c)$, and the ordering in the quotient $k_c^{\times}/(O^{\times}/(1+\mathcal{M}_c)\subset \Gamma$.

In this language we may consider
the expansion of the language by a group section of
$RV_c\to \Gamma_c$. Using this section we can identify 
$RV_c$ with $k_c^{\times}\times \Gamma_c$. An application of Lemma \ref{orthogonal}, 
shows that $k_c$ and $\Gamma_c$ are orthogonal;
and $0$-definable
subsets of $\Gamma_c^n$ 
are $0$-definable in the ordered group language of $\Gamma_c$.

So we just need to prove that $RV_{p,\bullet}$ is interpretable in this 
language. 
As a set, $RV_n,k^{\times}$ and $\Gamma$ are interpretable in 
$\mathcal{L}_1$ by the quotients
 $RV_n=RV_c/((1+p^n\mathcal{M})/(1+\mathcal{M}_c))$, 
$k^{\times}=\mathcal{O}^{\times}/(1+\mathcal{M}_c)/((1+p^n\mathcal{M})/(1+\mathcal{M}_c))$, and 
$\Gamma=RV_c/(\mathcal{O}^{\times}/(1+\mathcal{M}_c))$ 

Also, from Lemma \ref{rv0}, we get that $\oplus$ and $\ominus$ in $RV_n$
are interpretable in $\mathcal{L}_1$. Finally note that if 
$I$ denotes the image of $k_c^{\times}$ in $\Gamma$ and 
$p:\Gamma\to \Gamma_c$ is the projection, then 
for $r\in \Gamma$ one has $0<r$ if and only if either $0<p(r)$, or 
$r\in I$ and $0<r$.
As the orders in $\Gamma_c$ and $I$ are definable in $\mathcal{L}_1$
we conclude that the order in $\Gamma$ is also definable.
\end{proof}
\begin{proposition}\label{p-to-0-closed-bounded}
Suppose $K$ is 1-h-minimal valued field of mixed characteristic
in some language $\mathcal{L}$. Suppose $K$ has algebraic $RV_{p,\bullet}$.

Suppose $K$ comes with a non-trivial residue characteristic $0$ coarsening
and let $\mathcal{L}_c$ be 
the extension of the language $\mathcal{L}$ obtained
by adding this.

Then $K_c$ in the language $\mathcal{L}_c$ satisfies conditions 3 and 4 in
Situation \ref{closed-bounded-0}.
\end{proposition}
\begin{proof}
We can add some constants to the Basarab language in $RV_{p,\bullet}$ and assume
that every set definable without parameters in $RV_{p,\bullet}$ in the
language $\mathcal{L}$ is definable without parametes in the $RV_{p,\bullet}$. 

 Denote $\mathcal{L}_e$ the extension of $K$ to the
coarsening and constants as before.
Denote $\mathcal{L}_{be}^0$ 
the Basarab language in $RV_c$ together with the maps
$RV_c\to RV_n$, the Basarab language in $RV_{p,\bullet}$, 
and some constants as in the first paragraph. 
We denote $\mathcal{L}_{be}^p$ the Basarab language with constants in 
$RV_{p,\bullet}$, together with
 a sort for $\Gamma_c$ and the function symbol for
the projection
$\Gamma\to \Gamma_c$.
Let $\mathcal{L}_{ve}$ be the valued field language of $v$ expanded by 
$v_c$ and the constants in $RV_{p,\bullet}$.

Note that 
$K_c$ is 1-h-minimal with the language $\mathcal{L}_e$, by Theorem 
2.2.7 of \cite{hensel-minII}.

We prove first that every  $X\subset RV_c\times RV_{p,\bullet}^r$ which is 
$\mathcal{L}_e$ definable, is definabe in $\mathcal{L}_{be}^0$.
Indeed, note that 
$v$ is 0-h-minimal in $\mathcal{L}_e$, because it is an $RV_{p,\bullet}$ 
expansion
(see Lemma 2.6.3 of \cite{hensel-minII}).
So there are $a_1,\dots,a_n\in K$ and $m$ an integer
which prepare the family $rv_c^{-1}(X_z)$ uniformly in $z\in RV_{p,\bullet}^r$,
see Proposition 2.3.2 of \cite{hensel-minII}.

So if $R=\{(rv_c(x),rv_c(x-a_1),\dots,rv_c(x-a_n)\}\subset RV_c\times RV_c^n$ 
then $X_z=R^{-1}(\pi_m^{-1}(Y_z))$ for an $\mathcal{L}_e$-definable family
$Y_z\subset RV_m^n$. Here we denote $\pi_m:RV_c^n\to RV_m^n$ the projection. 
By resplendent relative quantifier elimination
$\{Y_z\}_z$ is 
$\mathcal{L}_{be}^p$-definable, see Proposition \ref{resplendent-rv}. 
We conclude $X$ is $\mathcal{L}_{ve}$-definable, 
and by resplendent quantifier elimination relative to $RV_c$ 
(in the valued field language of $v_c$), 
$X$ is $\mathcal{L}_{be}^0$-definable.

Now we 
show that in $\mathcal{L}_e$ $\Gamma_c,k_c$ satisfy property 3 of 
Situation \ref{closed-bounded-h0}.
As $k_c\subset RV_c$ and $\Gamma_c$ is 
interpretable in $RV_{p,\bullet}$, by what we proved
before we see that it is
enough to see that $k_c$ and $\Gamma_c$ are orthogonal in the language 
$\mathcal{L}_{be}^0$. This follows from Lemma \ref{orthogonal-p}. 

We still have to see property 4 of Situation \ref{closed-bounded-h0}. 
Let  $X\subset \Gamma_c^n$ be
definable in $\mathcal{L}_e$.
We have to see it is definable in the ordered group language in
$\Gamma_c^n$. By resplendent relative 
quantifier elimination it is definable in 
$\mathcal{L}_{be}^p$ in $RV_{p,\bullet}$, and so a fortiori it is definable
in $\mathcal{L}_{be}^0$.
By Lemma \ref{orthogonal-p}
we finish.
\end{proof}

Question: Suppose $K$ is a 0-h-minimal field in mixed characteristic
with a language $\mathcal{L}$
that extends the valued field language,
and expand the language to $\mathcal{L}_c$
to include a non-trivial residue characteristic $0$ coarsening
$v_c$.
Does it follow that every $0$-definable set $X\subset RV_c^n$ is definable 
in the expansion of the Basarab language in $RV_c$ described in 
Lemma \ref{orthogonal-p} and the  
$0$-$\mathcal{L}$-definable sets
in $RV_{p,\bullet}$?. In the course of the above proof we see this for $n=1$.

\begin{proposition}\label{closed-bounded-p}
Suppose $K$ is a 1-h-minimal valued field of mixed characteristic
with algebraic $RV_{p,\bullet}$.
Then $K$ satisfies the hypothesis in Situation \ref{closed-bounded-h1}.
\end{proposition}
\begin{proof}
Let $K_c$ be 
an expansion of $K$ obtained by adding a non-trivial residue
characteristic $0$ coarsening to $K$. 
By Proposition \ref{p-to-0-closed-bounded}
$K_c$ satisfies the hypothesis in Proposition \ref{closed-bounded-0}.
If $\{C_r\}_{r\in\Gamma}$ is a decreasing $\mathcal{L}$-definable family
of closed bounded nonempty sets in $K$, and $p:\Gamma\to \Gamma_c$ is the 
projection, then
the family $D_s=\cap\{C_r\mid p(r)=s\}$ is definable in $\mathcal{L}_c$,
it is closed in $K_c$ because the topologies coincide, it is bounded
for the same reason, and it is non-empty because for $r\in \Gamma$ with
$p(r)>s$ one has $C_r\subset D_s$. We conclude that $D_s$ and so $C_r$ have
a non-empty intersection.
\end{proof}
As a consequence of Propositions \ref{closed-bounded-0}, \ref{closed-bounded-p}, and \ref{closed-bounded-0.5} we have the main result of this document.
\begin{proposition}\label{main-compactness}
Suppose $K$ is a 1-h-minimal field of residue characteristic $0$
with algebraic $RV$,
or a 1-h-minimal field of mixed characteristic $(0,p)$ with algebraic
$RV_{p,\bullet}$. 
Then every decreasing definable family $\{C_r\}_{r\in\Gamma}$
of closed bounded sets in $K^n$, has a nonempty intersection.
\end{proposition}
\section{Local groups contain definable open subgroups}\label{seclocgro}
Here we show that a definable topological group has a definable family of open
subgroups that form a neighborhood base of the identity, for general classes
of valued fields

\begin{situation}\label{compact}
Here $K$ is either a 
1-h-minimal valued field of residue characteristic $0$ with algebraic $RV$,
or 1-h-minimal of mixed characteristic with algebraic $RV_{p,\bullet}$,
or a definably spherically complete 
C-minimal expansion of ACVF$_{p,p}$ with algebraic $RV$.
\end{situation}

\begin{proposition}\label{local-groups-are-locally-closed}
Suppose $K$ is 1-h-minimal.
Suppose $X$ is a definable topological 
local group with the underlying topology being the subspace topology of $K^n$ in $X$,
then $X$ is isomorphic as a definable local group to a $Y\subset K^n$ which is closed
in $K^n$.
\end{proposition}
\begin{proof}
If $X$ is a definable local group we can choose $U$ definable open such that 
$U\cap X$ is symmetric and transitivity makes sense and holds.

Denote $C$ the closure of $U\cap X$. By Proposition 2.17
and Fact 6.7 in \cite{AcostaHasson} 
(applied to $X\cap U\subset C$)
we obtain that $U\cap X$ contains a nonempty set
relatively open in $C$, say $W\cap C$ for some open $W$ open.
Making $W$ smaller we may assume it is also closed.
If $g\in W\cap C$ then left multiplication by $g$ is a topological isomorphism
$L_g:g^{-1}(W\cap C)\to W\cap C$, and $g^{-1}(W\cap C)$ is open in $X$ around $0$.

Transferring the local group group structure from $g^{-1}(W\cap C)$ to $W\cap C$ via
$L_g$ we finish.
\end{proof}
As a remark we note that $X$ itself is locally closed around $0$ whenever
the closedness theorem applies, as then $g^{-1}(W\cap C)$, being the 
continuous image by a definable map
of a closed bounded set, is closed and bounded.
\begin{proposition}
Suppose $K$ is as in Situation \ref{compact}.
Suppose $X\subset K^n$ is a definable set, and in $X$ with the 
subspace topology there is the structure of a definable local topological
group.
Then there is a decreasing definable family $\{G_r\}_r$ of open subgroups of $X$
indexed by $\Gamma_{>0}$ which form a neighborhood basis of the identity.
\end{proposition}
\begin{proof}
We start by finding a definable open subgroup in $X$, and then show 
the stronger statement.

By Proposition \ref{local-groups-are-locally-closed} we may assume $X$ is closed.
Making $X$ smaller we may assume $X$ is closed and bounded.
Let $U\subset X$ be a closed and open definable subset of $X$ containing
the identity, and on
on which transitivity makes sense and holds, and similarly for inverses and
the identity element.
In that case we define $V$ to be the set of $a$ in $U$ such that 
$aU\subset U$. Then we have that $V$ contains $1$ and is closed
under products. Also $V$ is open in $X$ becase of Proposition \ref{wallman}. 
Then we have that $G=V\cap V^{-1}$ is an open subgroup of $X$.

Now if $B_r\subset U$ is a definable family of open and closed subsets of 
$X$ indexed by $r\in \Gamma_{\geq 0}$, which are bounded and decreasing, then 
we take $V_r$ to be the set of $a\in X$ such that 
$aB_s\subset B_s$ for all $s\leq r$. To see that this is an open set 
consider $S=\{(x,y)\in K^{\times}\times X\mid y\in B_{v(x)}\}$, and 
$S_r=S\cap\{(x,y)\mid v(x)\leq r\}$.
Then it becomes a straightforward exercise to see that $S$
is open and closed in $K^{\times}\times X$, and bounded in $K\times K^n$.
If we consider the continuous function 
$f:U\times K^{\times}\times  U\to K^{\times}\times X$
defined by $(y,x,y')\mapsto (x,yy')$, then $b\in V_r$ if and only if
$f(\{b\}\times S_r)\subset S_r$, and so by Proposition \ref{wallman} there is 
is an open ball around $b$ contained in $V_r$.
We conclude $G_r=V_r\cap V_r^{-1}$ is a decreasing definable 
family of open sets indexed
by $r$ as required.
\end{proof}

\appendix
\section{Ind-definable sets}\label{ind-definable}
Given $S$ a linear order definable in some model,
we define $\bar{S}$ to be set of definable Dedekind cuts of $S$,
together with
a maximal or minimal element if $S$ does not have it.

More precisely $x\in\bar{S}$ if $x\subset S$ is a definable
set satisfying:
\begin{enumerate}
\item If $ a\in x$ and $b<a$ then $b\in x$.
\item If $x$ has a supremum $a$ in $S$, then $a\in x$.
\end{enumerate}
If $x\in \bar{S}$ and $y\in\bar{S}$ we say $x\leq y$ if $x\subset y$.

This relation gives a linear order in $\bar{S}$.
With this linear order the set $S$ becomes the maximum
of $\bar{S}$. Also $\emptyset$ is a minimum of $\bar{S}$ if $S$ has 
no minimal element, and if $S$ has a minimal element $m$, then
 $\{m\}$ is the minimum of $\bar{S}$.

We consider the map 
$\iota:S\to \bar{S}$ given by $a\mapsto \{b\in S\mid b\leq a\}$.
This is an injection and $\iota(a)<\iota(b)$ if and only if
$a<b$. We abuse notation and denote $S\subset \bar{S}$ via this map.
\begin{lemma}\label{completion}
If $X\subset S$ is a definable subset then $\text{sup}(X)$ and 
$\text{inf}(X)$ exist in $\bar{S}$.

If $x\in \bar{S}$, and $s\in S$, then $s\in x$ if and only if
$s\leq x$.

If $x,y\in \bar{S}$ and $x<y$, then there are $s,s'\in S$ such that 
$x\leq s<s'\leq y$.
\end{lemma}
The proof follows easily from the definitions and it is omitted.

Note that when $X\subset S$ is not bounded above, the supremum of $X$
is the maximum of $\bar{S}$, 
and when $X$ is empty the supremum is the the minimum
of $\bar{S}$. Similarly, when $X$ is unbounded below the infimum is 
the minimum of $\bar{S}$, and when $X$ is empty the infimum is the maximum
of $\bar{S}$.

We will use the notion of an ind-definable set.
The material following in this appendix is an elaboration of remark 2.2.9 of
\cite{HrLoe}. This remark says
that an ind-definable set is a directed colimit of definable sets.
These objects can be seen as 
sets with extra structure, and given an explicit description of the
underlying set there is in practice only one reasonable way of giving
the set the structure of an ind-definable set. This slight ambiguity
can be removed by using the Yoneda embedding functor. 
Next we provide the details of this construction.

Given a model $M$ in a possibly multisorted language, we consider the 
category of all $\emptyset$-definable sets in $M$ with $\emptyset$-definable
maps, and denote it $\text{Def}_0$. 
This category comes equipped with a forgetful functor
$\text{Fr}:\text{Def}_0\to Sets$ given by $D\mapsto D(M)$, which is faithful.
In this situation one calls $\text{Def}_0$ a category of sets with extra structure,
the underlying set of an object $D$ of $\text{Def}_0$ is $\text{Fr}(D)$, and it is 
often denoted by the 
same symbol as $D$, and informally identified with $D$,
 when there is no risk of confusion.
Also, given two $\emptyset$-definable sets $D_1,D_2$ a function of the 
underlying sets $f:Fr(D_1)\to Fr(D_2)$ is said to be $\emptyset$-definable
if there is an $\emptyset$-definable map $g:D_1\to D_2$ such that 
$\text{Fr}(g)=f$. In this case $f$ and $g$ will be denoted by the same symbol
and informally identified when there is no risk of confusion.

An $\emptyset$-intepretable set in $M$ is a set of the form $Y/E$ where 
$Y$ is $\emptyset$-definable and $E\subset Y\times Y$ is an
 $\emptyset$-definable equivalence relation.
The $\emptyset$-interpretable sets form a category
$\text{Int}_0$, with a forgetful functor $\text{Fr}:\text{Int}_0\to Sets$.
We have a fully faithful inclusion functor $\text{Def}_0\to \text{Int}_0$ which 
commutes with the forgetful functors.

When the model $M$ does not have $\emptyset$-definable elements in every sort,
it is more convenient to work in a category larger than $\text{Int}_0$ where there
are finite disjoint unions.
We define this category $\text{fInt}_0$ next. Its objects are the 
finite sequences $(X_k)_{k<n}$ with $X_k$ an $\emptyset$-interpretable object.
This object is denoted suggestively as $\bigsqcup_{k<n}X_k$.
Given $\emptyset$-interpretable sets $X$ and $Y_k$,
a morphism $X\to (Y_k)_{k<n}$ is given by a partition of $X$ into
$\emptyset$-interpretable subsets 
$X_k\subset X$ and a collection of $\emptyset$-definable maps
$X_k\to Y_k$. A morphism $(X_k)_{k<n}\to Y$ where $Y$ is an 
$\emptyset$-f-interpretable object, is a collection of morphisms
$X_k\to Y$. We omit the definition of the composition and the verification
that this gives a category. Note that $\text{Int}_0$ has a final object given by
$X/E$ where $E=X^2$, and $X$ is arbitrary, so it is not necessary to add
a formal final object to $\text{fInt}_0$.

We have a fully faithful functor 
$\text{Int}_0\to \text{fInt}_0$ and a forgetful faithful functor
$\text{fInt}_0\to Sets$ which commutes with the forgetful functor in $\text{fInt}_0$.
\begin{proposition}\label{quotient}
Suppose $X$ is an object in $\text{fInt}_0$. Suppose $E\subset X^2$ is an
$\emptyset$-definable subset, which is an equivalence relation.

Then there exist a quotient object $X/E$ in $\text{fInt}_0$.
Satisfying $\text{Fr}(X/E)=\text{Fr}(X)/\text{Fr}(E)$, and
characterized by the natural bijection
$\text{Mor}(X/E,Z)=\text{Mor}(X,Z)\times_{\text{Mor}(\text{Fr}(X)/\text{Fr}(E),\text{Fr}(Z))}\text{Mor}(\text{Fr}(X),\text{Fr}(Z))$.

In other words there is a canonical $\emptyset$-definable map 
$p:X\to X/E$ such that, $\text{Fr}(X/E)=\text{Fr}(X)/\text{Fr}(E)$ and $\text{Fr}(p)$
equals the canonical projection $\text{Fr}(X)\to \text{Fr}(X)/\text{Fr}(E)$.
And also, for every $f:X\to Z$ which is $\emptyset$-definable
and such that $\text{Fr}(f):\text{Fr}(X)\to \text{Fr}(Z)$ factors as
$g:Fr(X)/Fr(E)\to Fr(Z)$, there exists a unique $\bar{f}:X/E\to Z$
$\emptyset$-definable map such that $\bar{f}p=f$ and $\text{Fr}(\bar{f})=g$.
\end{proposition}
\begin{proof}
Let $X=\bigsqcup_k X_k$ and $E_k=E\cap X_k\times X_k$. Let
$Y_k=X_k/E_k$ which is an object in $\text{Int}_0$.
Then as sets $X/E=\cup_k Y_k$ (not necessarily disjoint union).
We take $Z_k\subset Y_k$ the $\emptyset$-definable subset given by 
$Y_k\setminus \cup_{r<k}Y_r$.
Then $X/E=\bigsqcup_kZ_k$. Some details omitted.
\end{proof}
Now we define the category of $\emptyset$-ind-f-interpretable sets
as the ind-completion of the category $\text{fInt}_0$.
So its objects are given by diagrams $(X_i)_{i\in I}$ in $\text{fInt}_0$
indexed by a directed set $I$, and denoted suggestively 
$\text{colim}_iX_i$, and the morphisms are given by 
the formula 
$\text{Mor}(\text{colim}_iX_i,\text{colim}_jY_j)=\text{lim}_i\text{colim}_j\text{Mor}(X_i,Y_j)$.
We get a fully faithful functor $\text{fInt}_0\to \text{indfInt}_0$
given by the constant system, and a forgetful functor
$\text{Fr}:\text{indfInt}_0\to Sets$. The forgetful functor acts on objects by the formula
$\text{Fr}(\text{colim}_iX_i)=\text{colim}_iFr(X_i)$, 
 and it is faithful when $M$ is $\omega$-saturated.

From here on out we assume $M$ is $\omega$-saturated.

The category $\text{indfInt}_0$ has directed colimits and $\text{Fr}$ commutes with them.
The category $\text{fInt}_0$ has finite limits (e.g., products of two elements and
fiber products) and $\text{Fr}$ commutes with them.
From this one gets that $\text{indfInt}_0$ has finite limits and $\text{Fr}$ 
commutes with them.

The category $\text{fInt}_0$ has finite coproducts and $\text{Fr}$ commutes with them.
This implies that $\text{indfInt}_0$ has finite coproducts and $\text{Fr}$ commutes
with them.

\begin{proposition}\label{yoneda-1}
If we denote $F(C,Sets)$ the category of contravariant functors
$C\to Sets$, then the functor
$Y:\text{indfInt}_0\to F(\text{Int}_0,Sets)$ which is the restriction of the Yoneda functor,
that is, the functor which acts on objects as $Y(X)(D)=\text{Mor}(D,X)$, is fully
faithful.
\end{proposition}
\begin{proof}
We have the Yoneda embedding 
$Y_0:\text{indfInt}_0\to F(\text{indfInt}_0,Sets)$ and the restriction functors
$r_0:F(\text{indfInt}_0,Sets)\to F(\text{fInt}_0,Sets)$, and 
$r_1:F(\text{fInt}_0,Sets)\to F(\text{Int}_0,Sets)$. 

We have that $Y_0$ is a fully faithful functor by the Yoneda lemma.
The idea is that 
we can recover a representable functor by its action on $\text{Int}_0$.
More precisely one gets functors 
$s_1:F(\text{Int}_0,Sets)\to F(\text{fInt}_0,Sets)$, 
$s_0:F(\text{fInt}_0,Sets)\to F(\text{indfInt}_0,Sets)$, such that
$r_is_i$ is naturally equivalent to the identity and 
$s_ir_i$ restricted to the essential image of $\text{indfInt}_0$ is naturally
equivalent to the identity.
These functors are given by the following formulas on objects:
$s_1F(\bigsqcup_kX_k)=\Pi_kF(X_k)$
and $s_0F(\text{colim}_i X_i)=\text{lim}_iF(X_i)$. From here the rest of the verification
is straightforward.
\end{proof}
Given an $\emptyset$-ind-f-interpretable set $X$, we can also consider
$\text{Mor}(D,X)$ as a functor on $D$ running over the $\emptyset$-definable sets.
Note that we have that $\text{Mor}(D,X)\subset \text{Mor}(Fr(D),Fr(X))$.
Let us denote $\mathcal{G}(\text{Def}_0)$ the category of pairs $(F,Z)$,
where $Z$ is a set, and $F$ is a contravariant 
functor $F:\text{Def}_0\to Sets$ which is a 
subfunctor of $\text{Mor}(Fr(-),Z)$.
Then we have the Yoneda functor
$\text{indfInt}_0\to \mathcal{G}$ given by $X\mapsto (\text{Mor}(D,X),Fr(X))$.
\begin{proposition}\label{yoneda-2}
With the above notation
the Yoneda functor $\text{indfInt}_0\to \mathcal{G}(\text{Def}_0)$ is fully faithful.
\end{proposition}
A pair $(F,P)$ is called representable if it is in the essential image
of the functor described above. The upshot of this proposition is 
that to give an $\emptyset$-ind-f-interpretable structure on the set
$P$ is the same thing as finding a subfunctor $F$ of $\text{Mor}(Fr(-),P)$
such that $(F,P)$ is representable.
\begin{proof}
That the Yoneda functor $\text{indfInt}_0\to \mathcal{G}(\text{Int}_0)$ is fully
faithful follows from Proposition \ref{yoneda-1}.
We have a restriction functor $r:\mathcal{G}(\text{Int}_0)\to \mathcal{G}(\text{Def}_0)$.
The idea is to find a functor 
$s:\mathcal{G}^r\to \mathcal{G}(\text{Int}_0)$ which is a
quasi-inverse when restricted to representable functors.
Here $\mathcal{G}^r$ is a fully faithful subcategory of $\mathcal{G}(\text{Def}_0)$ 
which contains the 
representable functors.
This $s$ is given in objects by 
$s(F,Z)(X/E)=\{f\in F(X)\mid f:Fr(X)\to Z\text{ factors through } Fr(X/E)\}$. 
In order to define $s(F,Z)$ on interpretable maps we restrict to 
$\mathcal{G}^r$ the full subcategory of pairs $(F,P)$ such that 
for every $g:D_1\to D_2$ a surjective map of $\emptyset$-definable sets,
if $f:Fr(D_2)\to P$ is such that $fFr(g)\in F(D_1)$, then 
$f\in F(D_2)$.

Then, if $X/E_1$ and $Y/E_2$ are objects of $\text{Int}_0$ and 
$f:Y/E_2\to X/E_1$ is an $\emptyset$-definable map, we consider
$R\subset Y\times X$ given by $R=\{(y,x)\mid p(x)=fp(y)\}$.
We have the coordinate projection maps $a:R\to Y$ and $b:R\to X$.
Note that $a$ is surjective.
The map $s(F,Z)(f):s(F,Z)(X/E_1)\to s(F,Z)(Y/E_2)$
is defined by $s(F,Z)(f)(r_1)=r_2$ satisfying 
$F(a)(r_2)=F(b)(r_1)$. 
We omit the strightforward but cumbersome verification that this is a 
well-defined function, that $s(F,Z)$ is a functor, that $s$ is a functor
and that $r$ and $s$ are quasi-inverses when restricted to representable
objects.
\end{proof}
Next we give a criterion for representability.
\begin{proposition}\label{representable}
With the notation of Proposition \ref{yoneda-2},
assume that $(F,P)$ is an object in $\mathcal{G}(\text{Def}_0)$, satisfying
the following conditions:
\begin{enumerate}
\item For every $p\in P$ there is $D$ in $\text{Def}_0$ an a 
$f\in F(D)$ such that $p\in \text{Im}(f)$.
\item For every $g:D_1\to D_2$ a surjective map of $\emptyset$-definable sets,
if $f:D_2\to P$ is such that $fg\in F(D_1)$, then 
$f\in F(D_2)$.
\item If $D=\sqcup_kD_k$ is a partition of the $\emptyset$-definable set $D$
into a finite number of
 $\emptyset$-definable subsets, and $f:D\to P$ is such that 
$f|_{D_i}\in F(D_i)$, then $f\in F(D)$.
\item For every $D$ in $\text{Def}_0$ and $f,g\in F(D)$, the set
$\{x\in D\mid f(x)=g(x)\}$ is a definable subset of $D$.
\end{enumerate}
Then $(F,P)$ is representable.
\end{proposition}
\begin{proof}

For $D\in \text{Def}_0$ and $h\in F(D)$, define 
$E_h\subset D\times D$ as the set of pairs $(x,y)$ such that 
$h(x)=h(y)$. This is an equivalence relation on $D$ and it is definable
by item 4.

We consider 
$\mathcal{I}$ the category of pairs $(D,h)$ with $h\in F(D)$ and morphisms
$(D_1,h_1)\to (D_2,h_2)$ given by a map $g:D_1\to D_2$ such that
$h_2Fr(g)=h_1$.
Now consider the category $\mathcal{J}$ 
with objects the finite subsets of $\mathcal{I}$
and morphisms $s\to t$ given by a function of sets $l:s\to t$
and a collection of morphisms in $\mathcal{I}$, $(D,h)\to l(D,h)$.
For $s\in \mathcal{J}$ define $D_s=\bigsqcup_{(D,h)\in s}D$.
For $l:s\to t$ a morphism in $\mathcal{J}$, 
we also get a morphism $D_l:D_s\to D_t$,
in other words $D_{\bullet}$ is a functor $\mathcal{J}\to \text{fInt}_0$.

Consider the map $h_s:D_s\to P$ given by $h_s\tau_{(D,h)}=h$,
where $\tau_{(D,h)}$ is the canonical coprojection $D\to D_s$.
Note that $h_sD_g=h_t$ for every $g:s\to t$ morphism.
Consider $E_s\subset D_s\times D_s$ given by 
$(x,y)\in E_s$ if and only if $h_s(x)=h_s(y)$.
This is a definable subset of $D_s\times D_s$ by item 4.
Indeed that it is definable is equivalent to 
$\{(x,y)\in D_1\times D_2\mid h_1(x)=h_2(y)\}$ being definable in 
$D_1\times D_2$, for $D_i$ definable and $h_i\in F(D_i)$, which follows
from item 4.
This implies $D_s/E_s$ is an element of $\text{fInt}_0$ which we denote $X_s$,
see Proposition \ref{quotient}.

The map $h_s$ factors as an injection $\bar{h}_s:X_s\to P$.
If $g:s\to t$ is a morphism in $\mathcal{J}$, 
then $D_g:D_s\to D_t$ factors as
$X_g:X_s\to X_t$, and as $\bar{h}_t$ is injective this does not depend
on the choice of $g$.
So if we define a relation in $\mathcal{J}$ by 
$s\leq t$, if there is a morphism
$s\to t$, this becomes a preorder in the set of objects of $\mathcal{J}$.
Denote $J$ the set of objects of $\mathcal{J}$ considered
as the category associated to this preorder. Then we get a diagram.
$X:J\to \text{fInt}_0$.
Consider $J'$ the partial order associated to $J$, note that the natural
projection $J\to J'$ is an equivalence of categories, 
so if we choose a quasi-inverse we get a diagram $X:J'\to \text{fInt}_0$.
Note also that $J'$ is a directed set. 

So we have that $X\in \text{indfInt}_0$.
We claim that $X$ represents the data $(F,P)$.
Note that the maps $\bar{h}_s:X_s\to P$ glue to a map 
$h:X\to P$ which is surjective by item 1, and injective by construction.

Now we have to see that if $D$ is a definable set, then the composition by
$h$ gives a bijection $\text{Mor}(D,X)\to F(D)$.
In other words $f\in F(D)$ if and only if $hf'=f$ for some ind-definable map
$f':D\to X$.
Take first $f\in F(D)$. Then $(D,f)\in \mathcal{I}$ 
and so we may take $s\in J'$ with 
$X_s=D/E_f$. We take $f'$ given by the projection 
$D\to X_s$ followed by the canonical map $X_s\to X$.
Then it follows from the definitions that $hf'=f$.
Now take $f$ such that $f=hf_1$ for some ind-definable 
morphism $f_1:D\to X$.
We have to show $f\in F(D)$.
By the definitions this means that $f_1$ factors as $D\to X_s\to X$
for some $s\in J'$ and $X_s\to X$ the canonical map.
If we denote $f_2:D\to X_s$ we conclude that $f=\bar{h}_sf_2$.
We have that $s=\{(D_k,h_k)\}_k$ and we denote $X_k$ 
the image of $X_{(D_k,h_k)}$ in $X_s$. Then $X_s=X_1\cup\cdots\cup X_n$.
Restricting to $(f_2)^{-1}(X_k)$ and using item 3 we may assume 
$X_s=D_1/E_{h_1}$ for some $(D_1,h_1)\in I$.
Denote $R=\{(x,y)\in D\times D_1\mid f_2(x)=p(y)\}$, and 
$g_1:R\to D, g_2:R\to D_1$ the projections. Composing $f$ with $g_1$ and using
Item 2 we see that we may assume $f_2:D\to X_s$ lifts to a definable map
$f_3:D\to D_1$. Now denote $h_2=h_1f_3$, which is an element $F(D)$, as 
$F$ is a functor.
Then from the definitions 
$f=\bar{h}_sf_2=\bar{h}_spf_3=h_1f_3=h_2$ as required.
\end{proof}
One sees in the proof that a representing object of $(F,P)$ is given by 
a colimit $\text{colim}_iX_i$ where $X_i\in \text{fInt}_0$ and the intermediate maps
$X_i\to X_j$ are injective. Such an object is called 
strict $\emptyset$-ind-f-interpretable.
A representable pair $(F,P)$ always satisfies items 1,2 and 3, and it satisfies
item 4 if and only if it is representable by a strict 
$\emptyset$-ind-f-interpretable set.
\begin{example}\label{completion-ind-definable}
We apply this to the case we considered in the start of the section.
Take $\bar{S}$ the definable completion of a linear order $S$ definable
in some model. If $D$ is $\emptyset$-definable then 
$F(D)$ is defined as the set of functions $f:D\to \bar{S}$ such that 
$\{f(x)\}_{x\in D}$ is a $\emptyset$-definable family of subsets of $S$.
Then we see that the pair $(F,\bar{S})$ is representable by an element
of $\text{indfInt}_0$, by the Proposition \ref{representable} (all conditions
in the hypothesis are easy to check).

Now we can see that the inclusion map $S\to \bar{S}$ is ind-definable.
This follows as the family $\{\{x\in S\mid x\leq a\}\}_{a\in S}=\{i(a)\}_{a\in S}$
is an $\emptyset$-definable family.

\end{example}
\begin{example}\label{completion-function-definable}
Suppose $C\subset S$ is an $\emptyset$-definable subset.
Here we verify that $C_{-}:\bar{S}\to \bar{S}$ defined by the formula
$C_{-}(a)=\text{sup}\{x\in S\mid x<a, x\in C\}$ is a well defined ind-definable 
map.

Note that if $a\in \bar{S}$ then $\{x\in S\mid x<a,x\in C\}$ is definable
in $S$ so it has a supremum in $\bar{S}$.
To see that it is ind-definable one can use Proposition \ref{yoneda-2} 
for instance.
So one needs to verify that for $T$ definable and $T\to \bar{S}$
 an ind-definable map,
then $C_{-}T$ is ind-definable.
The ind-definable maps $T\to \bar{S}$ are in natural bijective correspondence
with the definable families $\{x_t\}_{t\in T}$ of cuts in $S$.
So now has to check that $\{C_{-}(x_t)\}_{t\in T}$ is also a definable family
of cuts in $S$. This is straigthforward, for example,
from Lemma \ref{completion}
$C_{-}(x_t)=\{s\in S\mid s\leq C_{-}(x_t)\}$, and 
$s\leq C_{-}(x_t)$ if and only if for every $s'<s$, $s'\in S$, we have
$s'<C_{-}(x_t)$, and this holds if and only 
there is a $s''\in S$ such that $s'<s''\leq x_t$ and $s''\in C$.
As the condition $s''\leq x_t$ can be expresed by a formula in $(s'',t)$ we
see that the condition $s\leq C_{-}(x_t)$ can be expressed by
a formula in $(s,t)$.

One can check in a similar way that $C_{+}:\bar{S}\to \bar{S}$
given by $C_{+}(a)=\text{Inf}\{x\in S\mid a<x, x\in C\}$ 
makes sense and is ind-definable.
\end{example}
\begin{definition}
If $X$ is an ind-definable set, then $Y\subset X$ is called 
$\emptyset$-relatively definable.
if for every $\emptyset$-ind-definable map $f:D\to X$, where $D$
is $\emptyset$-definable, the inverse image $f^{-1}(Y)\subset D$
is $\emptyset$-definable.

An element $x\in X$ is $\emptyset$-definable if $\{x\}$ is 
$\emptyset$-relatively definable
\end{definition}
Note that this definition can be read directly from the pair $(F,P)$ in
Proposition \ref{representable}. 
Indeed, $Q\subset P$ is $\emptyset$-relatively
definable if and only if for every $f\in F(D)$ one has $f^{-1}(Q)\subset D$
is $\emptyset$-definable.

Note also that boolean combinations of $\emptyset$-relatively 
definable subsets are
$\emptyset$-relatively definable.
\begin{example}
If $X$ is an $\emptyset$-ind-definable set, then $X$ is strict ind-definable
if and only if the diagonal in $X\times X$ is relatively definable.
This is a restatement of condition 4 in Proposition \ref{representable}.
\end{example}
\begin{example}\label{completion-order-definable}
In $\bar{S}$ the relation $\leq$ in $\bar{S}^2$ is 
$\emptyset$-relatively definable.
This boils down to verifying that if $\{X_a\}_{a\in T}$ and 
$\{Y_a\}_{a\in T}$ are $\emptyset$-definable families of cuts, then 
$\{a\in T\mid Y_a\leq X_a\}$ is a definable subset of $T$.

Similarly we have that the minimum and maximum of $\bar{S}$ are definable
elements of $\bar{S}$.
\end{example}
\bibliographystyle{plain}
\bibliography{harvard}
\end{document}